\documentclass[fleqn,twoside]{article}%
\usepackage{amsfonts}
\usepackage{espcrc1}
\usepackage{amsmath}
\usepackage{amssymb}
\usepackage{graphicx}
\usepackage{hyperref}
\usepackage[truetex]{color}%
\providecommand{\U}[1]{\protect\rule{.1in}{.1in}}
\newtheorem{theorem}{Theorem}

\newtheorem{definition}[theorem]{Definition}

\newtheorem{lemma}[theorem]{Lemma}

\newtheorem{remark}[theorem]{Remark}

\newenvironment{proof}[1][Proof]{\noindent\textbf{#1.} }{\ \rule{0.5em}{0.5em}}
\begin{document}
%

\title
{Partial-Approximate Controllability of Nonlocal Fractional Evolution Equations via Approximating Method}%
%

\author{N. I. Mahmudov\\{Eastern Mediterranean University\\
Gazimagusa Mersin 10, Turkey \\
email: nazim.mahmudov@emu.edu.tr}}%
%

\maketitle
%

\begin{abstract}
In this paper we study partial-approximate controllability of semilinear
nonlocal fractional evolution equations in Hilbert spaces. By using fractional
calculus, variational approach and approximating technique, we give the
approximate problem of the control system and get the compactness of
approximate solution set. Then new sufficient conditions for the
partial-approximate controllability of the control system are obtained when
the compactness conditions or Lipschitz conditions for the nonlocal function
are not required. Finally, we apply our abstract results to the
parial-approximate controllability of the semilinear heat equation and delay
equation.
\end{abstract}%

\section{Introduction}

Controllability concepts are important properties of a control system that
plays a substantial role in many engineering problems, such as stabilizing
unstable systems using feedback control. Therefore, in recent years,
controllability problems for various types of linear and semilinear
(fractional) dynamical systems have been studied in many articles (see
\cite{arti}-\cite{zuazua4} and the references therein). From the mathematical
point of view it is necessary to distinguish between problems of exact and
approximate controllability. Exact controllability allows to move the system
to an arbitrary final state, while approximate controllability means that the
system can be moved to an arbitrary small neighborhood of the final state. In
particular, approximately controlled systems are more predominant, and very
often the approximate controllability is quite adequate in applications. There
are some interesting and important approximate controllability results for
semilinear evolution systems in abstract spaces with a Caputo fractional
derivative. Sakthivel et. al. \cite{mah6} initiated to study approximate
controllability of fractional differential systems in Hilbert spaces.
Meanwhile, Sakthivel and Ren \cite{sakthivel}, Debbouche and Torres
\cite{nonloc7}, Mahmudov \cite{mah1}, \cite{mahACM2} pay attention to the
study of approximate controllability for various types of (fractional)
evolution systems in abstract spaces. On the other hand, existence and
approximate controllability of nonlocal fractional equations were studied in
\cite{nonloc5}, \cite{nonloc2}, \cite{nonlocal8}, \cite{mah2}

This paper is devoted to the study of partial-approximate controllability of
the semilinear fractional evolution equation with nonlocal conditions,%
\begin{align}
^{C}D_{t}^{q}y\left(  t\right)   &  =Ay\left(  t\right)  +Bu\left(  t\right)
+f\left(  t,y\left(  t\right)  \right)  ,\ 0<t\leq T,\nonumber\\
y\left(  0\right)   &  =y_{0}-g\left(  y\right)  , \label{cp1}%
\end{align}
where the state variable $y\left(  \cdot\right)  $ takes values in the Hilbert
space $X$, $^{C}D_{t}^{q}$ is the Caputo fractional derivative of order $q$
with $\frac{1}{2}<q\leq1,$ $A:D\left(  A\right)  \subset X\rightarrow X$ is a
family of closed and bounded linear operators generating a strongly continuous
semigroup $S:\left[  0,b\right]  \rightarrow\mathfrak{L}\left(  X\right)  $,
where the domain $D\left(  A\right)  \subset X$ is dense in $X,$ the control
function $u\left(  \cdot\right)  $ is given in $L^{2}\left(  \left[
0,b\right]  ,U\right)  ,$ $U$ is a Hilbert space, $B$ is a bounded linear
operator from $U$ into $X,$ $f:\left[  0,b\right]  \times X\rightarrow X,$
$g:C\left(  \left[  0,b\right]  ,X\right)  \rightarrow X$ are given functions
satisfying some assumptions to be specified later and $y_{0}$ is an element of
the Hilbert space $X$.

One of the purposes of this article is to investigate the partial-approximate
controllability of the system (\ref{cp1}) without Lipschitz continuous or
compact assumptions on the non-local term $g$. In fact, $g$ is supposed to be
continuous and is completely defined by $\left[  \delta,b\right]  $ for some
small $\delta>0$. Meanwhile, in order to obtain the existence of solutions of
the control system (\ref{cp1}), we construct an approximate problem of the
system (\ref{cp1}) and obtain the compactness of the set of approximate
solutions. It differs from the usual approach that the fixed-point theorem is
applied directly to the corresponding solution operator. Our results,
therefore, can be seen as the extension and development of existing results.

It has not yet been reported work on the partial-approximate controllability
of fractional semilinear evolution equations with nonlocal conditions.
Inspired by the aforementioned recent contributions, we propose to discuss the
partial-approximate controllability of fractional evolution systems in Hilbert
spaces with classical nonlocal conditions. We first impose a
partial-approximate controllability of the associated linear system. Then we
develop a variational approach in \cite{mah2}, \cite{mahACM2} and
approximating method in \cite{nonlocal8}, and rewrite our control problem as a
sequence of fixed point problems. Next using the Schauder fixed point theorem
we get the existence of fixed points and show that these solutions (fixed
points) steers the system to an arbitrary small neighborhood of the final
state in a closed subspace.

We organize the article as follows. In Section 2, we provide preliminaries,
assumptions and formulate the main result on partial-approximate
controllability. In Section 3, we extend the variational method to construct
approximating control for the approximating controllability problem and use
the Schauder fixed point theorem to show existence of a solution for the
approximating controllability problem. The partial-approximate controllability
result is proved in Section 4. In this section we solve the difficulty
concerning the compactness of solution operator $S\left(  t\right)  $ at $t=0$
by means of approximate solution set constructed in Section 3. The difficulty
is due to the fact that a compact semigroup $S\left(  t\right)  $ is not
compact at $t=0$ in an infinite dimensional space. Finally, we provide two
examples to illustrate the application of the abstract results.

\section{Statement of results}

Throughout this paper, let $\mathbb{N}$, $\mathbb{R}$, $\mathbb{R}_{+}$ be the
set of positive integers, real numbers and positive real numbers,
respectively. We denote by $X$ a Hilbert space with norm $\left\Vert
\cdot\right\Vert $, $C\left(  \left[  0,b\right]  ,X\right)  $ the space of
all $X$-valued continuous functions on $\left[  0,b\right]  $ with the norm
$\left\Vert \cdot\right\Vert _{C}$, $L^{2}(\left[  0,b\right]  ,U)$ the space
of all $U$-valued square integrable functions on $\left[  0,b\right]  $,
$\mathcal{L}\left(  X\right)  $ the space of all bounded linear operators from
$X$ to $X$ with the usual norm $\left\Vert \cdot\right\Vert _{\mathcal{L}%
\left(  X\right)  }$, let $A$ be the infinitesimal generator of $C_{0}%
$-semigroup $\left\{  S\left(  t\right)  :t\geq0\right\}  $ of uniformly
bounded linear operators on $X$. Clearly, $M_{S}:=\sup\left\{  \left\Vert
S\left(  t\right)  \right\Vert _{\mathcal{L}\left(  X\right)  }:t\geq
0\right\}  <\infty$. Let $E$ be a closed subspace of $X$ and denote by $\Pi$
the projection from $X$ onto $E$.

We recall some notations, definitions and results on fractional derivative and
fractional differential equations.

\begin{definition}
\label{def:1}\cite{pod} The Riemann-Liouville fractional order derivative of
$f:\left[  0,\infty\right)  \rightarrow X$ of order $q\in\mathbb{R}_{+}$ is
defined by%
\[
^{RL}D^{q}f\left(  t\right)  =\frac{1}{\Gamma\left(  n-q\right)  }\frac{d^{n}%
}{dt^{n}}\int_{0}^{t}\left(  t-s\right)  ^{n-q-1}f\left(  s\right)  ds,
\]
where $q\in\left(  n-1,n\right)  $, $n\in\mathbb{N}$.
\end{definition}

\begin{definition}
\label{def:2}\cite{pod} The Caputo fractional order derivative of $f:\left[
0,\infty\right)  \rightarrow X$ of order $q\in\mathbb{R}_{+}$ is defined by%
\[
^{RL}D^{q}f\left(  t\right)  =\frac{1}{\Gamma\left(  n-q\right)  }\int_{0}%
^{t}\left(  t-s\right)  ^{n-q-1}f^{\left(  n\right)  }\left(  s\right)  ds,
\]
where $q\in\left(  n-1,n\right)  $, $n\in\mathbb{N}$.
\end{definition}

Now, we use the probability density function to give the following definition
of mild solutions to (\ref{cp1}).

\begin{definition}
\cite{zhou1} A solution $y(\cdot;u)\in C(\left[  0,b\right]  ,X)$ is said to
be a mild solution of (\ref{cp1}) if for any $u\in L^{2}(\left[  0,b\right]
,U)$ the integral equation
\begin{equation}
y(t)=S_{q}\left(  t\right)  \left(  y_{0}-g\left(  y\right)  \right)
+\int_{0}^{t}\left(  t-s\right)  ^{q-1}T_{q}\left(  t-s\right)  \left[
Bu\left(  s\right)  +f\left(  s,y\left(  s\right)  \right)  \right]
ds,\ \ 0\leq t\leq b, \label{cs2}%
\end{equation}
is satisfied. Here%
\begin{align*}
S_{q}\left(  t\right)   &  =\int_{0}^{\infty}\omega_{\alpha}\left(
\theta\right)  S\left(  t^{\alpha}\theta\right)  d\theta,\ \ \ T_{q}\left(
t\right)  =\alpha\int_{0}^{\infty}\theta\omega_{\alpha}\left(  \theta\right)
S\left(  t^{\alpha}\theta\right)  d\theta,\ \ \ t\geq0,\\
\omega_{\alpha}\left(  \theta\right)   &  =\frac{1}{\alpha}\theta
^{-1-1/\alpha}\varpi_{\alpha}\left(  \theta^{-1/\alpha}\right)  \geq0,\\
\varpi_{\alpha}\left(  \theta\right)   &  =\frac{1}{\pi}\sum_{n=1}^{\infty
}\left(  -1\right)  ^{n-1}\theta^{-n\alpha-1}\frac{\Gamma\left(
n\alpha+1\right)  }{n!}\sin\left(  n\pi\alpha\right)  ,\ \ \ \theta\in\left(
0,\infty\right)  ,
\end{align*}
where is a probability density function defined on $\left(  0,\infty\right)
$, that is%
\[
\omega_{\alpha}\left(  \theta\right)  \geq0,\ \ \ \theta\in\left(
0,\infty\right)  ,\ \ \ \int_{0}^{\infty}\omega_{\alpha}\left(  \theta\right)
d\theta=1.
\]

\end{definition}

We present some basic properties of $S_{q}$ and $T_{q}$ which will be used in
the sequel \cite{zhou1}.

\begin{itemize}
\item For any fixed $t\geq0$ and any $y\in X$, $\left\Vert S_{q}\left(
t\right)  y\right\Vert \leq M_{S}\left\Vert y\right\Vert $ and $\left\Vert
T_{q}\left(  t\right)  y\right\Vert \leq\dfrac{M_{S}}{\Gamma\left(  q\right)
}\left\Vert y\right\Vert .$

\item $\left\{  S_{q}\left(  t\right)  :t\geq0\right\}  $ and $\left\{
T_{q}\left(  t\right)  :t\geq0\right\}  $ are strongly continuous.

\item $\left\{  S_{q}\left(  t\right)  :t>0\right\}  $ and $\left\{
T_{q}\left(  t\right)  :t>0\right\}  $ are compact operators provided that
$\left\{  S\left(  t\right)  :t>0\right\}  $ is compact.
\end{itemize}

Let $y(b;u)$ be the state value of (\ref{cp1}) at terminal time $b$
corresponding to the control $u$.

\begin{definition}
\cite{bm2} Given $b>0,$ $y_{0}\in X,$ $y_{b}\in E$ and $\varepsilon>0.$ The
system (\ref{cp1}) is said to be partial-approximately controllable on
$\left[  0,b\right]  ,$ if there exists a control $u_{\varepsilon}\in
L^{2}\left(  \left[  0,b\right]  ,U\right)  $ such that the corresponding
solution $y\left(  t;u_{\varepsilon}\right)  $ of (\ref{cp1}), satisfies the
conditions%
\begin{equation}
\left\Vert \Pi y\left(  b;u_{\varepsilon}\right)  -y_{b}\right\Vert
<\varepsilon. \label{cp3}%
\end{equation}

\end{definition}

\begin{remark}
In particular, if $E=X$, the concept of partial-approximate controllability
coincide with the well known concept of approximate controllability.
\end{remark}

\begin{remark}
It is known that the system (\ref{ls}) approximately controllable on $\left[
0,b\right]  $ if and only if the condition $B^{\ast}T_{q}^{\ast}\left(
b-s\right)  \Pi^{\ast}\varphi=0,\ 0<s<b$ implies that $\varphi=0,$ see
\cite{curt}.
\end{remark}

Our main result is as follows:

\begin{theorem}
\label{thm:main}Assume the following conditions:

\begin{enumerate}
\item[(S)] $S\left(  t\right)  ,$ $t>0$ is compact operator;

\item[(F)] The function $f:\left[  0,b\right]  \times X\rightarrow X$
satisfies the following

\begin{enumerate}
\item $f\left(  \cdot,\cdot\right)  :X\rightarrow X$ is jointly continuous;

\item there is a positive continuous function $n\in C\left(  \left[
0,b\right]  ,\mathbb{R}_{+}\right)  $ such that for every $(t,y)\in\left[
0,b\right]  \times X$, we have
\[
\left\Vert f\left(  t,y\right)  \right\Vert \leq n\left(  t\right)  .
\]

\end{enumerate}

\item[(G)] The function $g:C\left(  \left[  0,b\right]  ,X\right)  \rightarrow
X$ satisfies the following

\begin{enumerate}
\item $g$ is continuous, there exists a positive constant $\Lambda_{g}$ such
that for all $y\in X$,%
\[
\left\Vert g\left(  y\right)  \right\Vert \leq\Lambda_{g}.
\]

\item There is a $\delta\in\left(  0,b\right)  $ such that for any $y,z\in
C\left(  \left[  0,b\right]  ,X\right)  $ satisfying $y\left(  t\right)
=z\left(  t\right)  $, $t\in\left[  \delta,b\right]  $, $g\left(  y\right)
=g\left(  z\right)  $.
\end{enumerate}

\item[(B)] $B:U\rightarrow X$ is a linear continuous operator with
$M_{B}:=\left\Vert B\right\Vert $.

\item[(AC)] The linear system%
\begin{equation}
y(t)=S_{q}\left(  t\right)  y_{0}+\int_{0}^{t}\left(  t-s\right)  ^{q-1}%
T_{q}\left(  t-s\right)  Bu\left(  s\right)  ds \label{ls}%
\end{equation}
is partial-approximately controllable in $\left[  0,b\right]  .$ \newline Then
the fractional control system (\ref{cs2}) is partial-approximately
controllable on $\left[  0,b\right]  $.
\end{enumerate}
\end{theorem}

\begin{remark}
Our results are new for approximate controllability, that is for the case
$E=X$.
\end{remark}

\begin{remark}
Our results are new even for the classical evolution control system with
nonlocal conditions (the case $q=1$).
\end{remark}

\begin{remark}
One may expect the results of this paper to hold for a class of problems
governed by different type of evolution systems such as Riemann-Liouville
FDEs, stochastic (fractional) DEs, FDEs with infinite delay and so on.
\end{remark}

\section{Auxiliary lemmas}

For $\mathbb{\varepsilon}>0$ and $n\geq1$ we introduce the following
functional%
\begin{equation}
J_{\mathbb{\varepsilon},n}\left(  \varphi;z\right)  =\frac{1}{2}\int_{0}%
^{b}\left(  b-s\right)  ^{q-1}\left\Vert B^{\ast}T_{q}^{\ast}\left(
b-s\right)  \Pi^{\ast}\varphi\right\Vert ^{2}ds+\mathbb{\varepsilon}\left\Vert
\varphi\right\Vert -\left\langle \varphi,h_{n}\left(  z\right)  \right\rangle
, \label{f1}%
\end{equation}
where
\[
h_{n}\left(  z\right)  =\Pi S_{q}\left(  b\right)  \left(  y_{0}-S\left(
\frac{1}{n}\right)  g\left(  z\right)  \right)  +\int_{0}^{b}\Pi T_{q}\left(
b-s\right)  f\left(  s,z\left(  s\right)  \right)  ds-y_{b}.
\]
and the following approximating operators $Q:C\left(  \left[  0,b\right]
,X\right)  \rightarrow E$
\begin{align*}
Q\left(  z\right)   &  :=\Pi S_{q}\left(  b\right)  \left(  y_{0}-S\left(
\frac{1}{n}\right)  g\left(  z\right)  \right)  +\int_{0}^{b}\left(
b-s\right)  ^{q-1}\Pi T_{q}\left(  b-s\right)  f\left(  s,z\left(  s\right)
\right)  ds\\
&  =:Q_{1}\left(  z\right)  +Q_{2}\left(  z\right)  .
\end{align*}
Set%
\[
B\left(  0;r\right)  =\left\{  x\in C\left(  \left[  0,b\right]  ,X\right)
:\left\Vert x\right\Vert _{C}\leq r\right\}  .
\]

\begin{lemma}
\label{lem:rc}The mapping $Q:B\left(  0;r\right)  \rightarrow E$ is compact.
\end{lemma}

\begin{proof}
For any $\mathbb{\eta}\in\left(  0,b\right)  $ and $\delta>0$, we define an
operator $Q_{2}^{\eta,\delta}$ on $B\left(  0;r\right)  $ by the formula%
\begin{align*}
Q_{2}^{\eta,\delta}\left(  z\right)   &  =q\int_{0}^{b-\mathbb{\eta}}%
\int_{\delta}^{\infty}\left(  b-s\right)  ^{q-1}\theta\omega_{q}\left(
\theta\right)  \Pi S\left(  \left(  b-s\right)  ^{q}\theta\right)  f\left(
s,z\left(  s\right)  \right)  d\theta ds\\
&  =\Pi S\left(  \mathbb{\eta}^{q}\delta\right)  q\int_{0}^{b-\mathbb{\eta}%
}\int_{\delta}^{\infty}\left(  b-s\right)  ^{q-1}\theta\omega_{q}\left(
\theta\right)  S\left(  \left(  b-s\right)  ^{q}\theta-S\left(  \mathbb{\eta
}^{q}\delta\right)  \right)  f\left(  s,z\left(  s\right)  \right)  d\theta
ds\\
&  =:\Pi S\left(  \mathbb{\eta}^{q}\delta\right)  Z\left(  \mathbb{\eta
}\right)  .
\end{align*}
Bt the assumption (F)%
\[
\left\Vert Z\left(  \mathbb{\eta}\right)  \right\Vert \leq bM_{S}\left\Vert
n\right\Vert _{C}.
\]
Then from the compactness of $S\left(  \mathbb{\eta}^{q}\delta\right)  \ $and
boundedness of $Z\left(  \mathbb{\eta}\right)  $, we obtain that the set
$Q_{2}^{\eta,\delta}:B\left(  0;r\right)  \rightarrow E$ is compact. Moreover,
for every $z\in B\left(  0;r\right)  $, we have%
\begin{align*}
&  \left\Vert Q_{2}\left(  z\right)  -Q_{2}^{\eta,\delta}\left(  z\right)
\right\Vert \\
&  \leq q\int_{0}^{b}\int_{0}^{\delta}\left(  b-s\right)  ^{q-1}\theta
\omega_{q}\left(  \theta\right)  \left\Vert \Pi S\left(  \left(  b-s\right)
^{q}\theta\right)  \right\Vert \left\Vert f\left(  s,z\left(  s\right)
\right)  \right\Vert d\theta ds\\
&  +q\int_{b-\mathbb{\eta}}^{b}\int_{\delta}^{\infty}\left(  b-s\right)
^{q-1}\theta\omega_{q}\left(  \theta\right)  \left\Vert \Pi S\left(  \left(
b-s\right)  ^{q}\theta\right)  \right\Vert \left\Vert f\left(  s,z\left(
s\right)  \right)  \right\Vert d\theta ds\\
&  \leq\left\Vert \Pi\right\Vert M_{S}\left(  b^{q}\int_{0}^{\delta}%
\theta\omega_{q}\left(  \theta\right)  d\theta+\frac{\eta^{q}}{\Gamma\left(
q+1\right)  }\right)  \left\Vert n\right\Vert _{C}\rightarrow0,
\end{align*}
as $\eta,\delta\rightarrow0.$Hence, there exist relatively compact operator
that can be arbitrarily close to $Q_{2}$. Then, $Q_{2}$ is compact.

Since $S\left(  \frac{1}{n}\right)  $, $n\geq1$, is compact, $\Pi$ is linear
bounded and the assumption (G) holds, it is easily seen that, the set
\[
\left\{  \Pi S_{q}\left(  b\right)  \left(  y_{0}-S\left(  \frac{1}{n}\right)
g\left(  z\right)  \right)  :z\in B\left(  0;r\right)  \right\}
\]
is relatively compact in $X$. Thus the operator $Q=Q_{1}+Q_{2}$ is compact.
\end{proof}

\begin{lemma}
\label{lem:4}For any $n\geq1$, $h_{n}:B\left(  0;r\right)  \rightarrow E$ is a
continuous function.
\end{lemma}

\begin{proof}
For any $z_{l},z\in B\left(  0;r\right)  $ with $\lim_{l\rightarrow\infty
}\left\Vert z_{l}-z\right\Vert _{C}=0.$ By the condition (F1), we have we can
conclude that $\lim_{l\rightarrow\infty}\left\Vert f\left(  \cdot,z_{l}\left(
\cdot\right)  \right)  -f\left(  \cdot,z\left(  \cdot\right)  \right)
\right\Vert _{C}=0.$ On the other hand%
\begin{align*}
\left\Vert h_{n}\left(  z_{l}\right)  -h_{n}\left(  z\right)  \right\Vert
_{E}  &  \leq\left\Vert \Pi S_{q}\left(  b\right)  S\left(  \frac{1}%
{n}\right)  \left(  g\left(  z_{l}\right)  -g\left(  z\right)  \right)
\right\Vert \\
&  +\left\Vert \int_{0}^{b}\left(  b-s\right)  ^{q-1}\Pi T_{q}\left(
b-s\right)  \left[  f\left(  s,z_{l}\left(  s\right)  \right)  -f\left(
s,z\left(  s\right)  \right)  \right]  d\theta ds\right\Vert \\
&  \leq M_{S}^{2}\left\Vert \Pi\right\Vert \left\Vert g\left(  z_{l}\right)
-g\left(  z\right)  \right\Vert \\
&  +\frac{M_{S}b^{q}}{\Gamma\left(  q\right)  }\left\Vert \Pi\right\Vert
\left\Vert f\left(  \cdot,z_{l}\left(  \cdot\right)  \right)  -f\left(
\cdot,z\left(  \cdot\right)  \right)  \right\Vert _{C},
\end{align*}
which implies that%
\[
\left\Vert h_{n}\left(  z_{l}\right)  -h_{n}\left(  z\right)  \right\Vert
_{E}\rightarrow0\ \ \ \text{as\ \ }l\rightarrow\infty.
\]

\end{proof}

From the definition of $\varphi\rightarrow J_{\mathbb{\varepsilon},n}\left(
\varphi;z\right)  $, we see immediately that for any fixed $z\in B\left(
0;r\right)  $ it is continuous and strictly convex.

\begin{lemma}
\label{lem:1}For any $B\left(  0;r\right)  $,%
\begin{equation}
\underline{\lim}_{\left\Vert \varphi\right\Vert \rightarrow\infty}\inf_{z\in
B\left(  0;r\right)  }\frac{J_{\mathbb{\varepsilon},n}\left(  \varphi
;z\right)  }{\left\Vert \varphi\right\Vert }\geq\mathbb{\varepsilon}.
\label{f7}%
\end{equation}

\end{lemma}

\begin{proof}
In order to prove (\ref{f7}), suppose that it is not the case. Then there
exists sequences $\left\{  \varphi_{l}\right\}  \subset X,$ $\left\{
z_{l}\right\}  \subset B\left(  0;r\right)  $, with $\left\Vert \varphi
_{l}\right\Vert \rightarrow\infty$, such that%
\begin{equation}
\underline{\lim}_{l\rightarrow\infty}\frac{J_{\mathbb{\varepsilon},n}\left(
\varphi_{l};z_{l}\right)  }{\left\Vert \varphi_{l}\right\Vert }%
<\mathbb{\varepsilon}. \label{f2}%
\end{equation}
Without loss of generality, we may assume that%
\begin{equation}
h_{n}\left(  z_{l}\right)  \rightarrow h_{n},\ \ \text{strongly}%
\ \text{in\ \ }X\text{,} \label{f3}%
\end{equation}
for some $h_{n}\in X.$ In fact, $\left\{  h_{n}\left(  z_{l}\right)
:l\geq1\right\}  \subset\operatorname{Im}Q$ is relatively compact in $X$;
thus, we may assume this by picking a subsequence.

Next, we normalize $\varphi_{l}:$ $\widetilde{\varphi}_{l}=\dfrac{\varphi_{l}%
}{\left\Vert \varphi_{l}\right\Vert }$. Since $\left\Vert \widetilde{\varphi
}_{l}\right\Vert =1,$ we can extract a subsequence (still denoted by
$\widetilde{\varphi}_{l}$), which weakly converges in $X$ to an element
$\widetilde{\varphi}$ in $X$. Consequently, because $S\left(  t\right)  ,$
$t>0$, is a compact semigroup, we see that
\begin{equation}
B^{\ast}S^{\ast}\left(  b-\cdot\right)  \Pi^{\ast}\widetilde{\varphi}%
_{l}\rightarrow B^{\ast}S^{\ast}\left(  b-\cdot\right)  \Pi^{\ast}%
\widetilde{\varphi},\ \ \ \text{strongly\ \ in}\ C\left(  \left[  0,b\right]
,X\right)  ,\ \text{as\ }l\rightarrow\infty. \label{f11}%
\end{equation}
From (\ref{f1}), it follows that%
\[
\frac{J_{\mathbb{\varepsilon},n}\left(  \varphi_{l};z_{l}\right)  }{\left\Vert
\varphi_{l}\right\Vert }=\frac{\left\Vert \varphi_{l}\right\Vert }{2}\int
_{0}^{b}\left(  b-s\right)  ^{q-1}\left\Vert B^{\ast}T_{q}^{\ast}\left(
b-s\right)  \Pi^{\ast}\varphi\right\Vert ^{2}ds+\mathbb{\varepsilon}\left\Vert
\widetilde{\varphi}_{l}\right\Vert -\left\langle \widetilde{\varphi}_{l}%
,h_{n}\left(  z_{l}\right)  \right\rangle .
\]
Thus, noting that $\left\Vert \varphi_{l}\right\Vert \rightarrow\infty,$ by
(\ref{f2})-(\ref{f11}) and the Fatou lemma%
\[
\int_{0}^{b}\left(  b-s\right)  ^{q-1}\left\Vert B^{\ast}T_{q}^{\ast}\left(
b-s\right)  \Pi^{\ast}\widetilde{\varphi}\right\Vert ^{2}ds\leq\underline
{\lim}_{l\rightarrow\infty}\int_{0}^{b}\left(  b-s\right)  ^{q-1}\left\Vert
B^{\ast}T_{q}^{\ast}\left(  b-s\right)  \Pi^{\ast}\widetilde{\varphi}%
_{l}\right\Vert ^{2}ds=0.
\]
By assumption (AC) we have $\widetilde{\varphi}=0,$ and we deduce that
\[
\widetilde{\varphi}_{l}\rightharpoonup0\ \ \text{weakly}\ \text{in
}X\ \text{as }l\rightarrow\infty.
\]
Hence%
\[
\underline{\lim}_{l\rightarrow\infty}\frac{J_{\mathbb{\varepsilon},n}\left(
\varphi_{l};z_{l}\right)  }{\left\Vert \varphi_{l}\right\Vert }\geq
\underline{\lim}_{l\rightarrow\infty}\left(  \mathbb{\varepsilon}\left\Vert
\widetilde{\varphi}_{l}\right\Vert -\left\langle \widetilde{\varphi}_{l}%
,h_{n}\left(  z_{l}\right)  \right\rangle \right)  =\mathbb{\varepsilon},
\]
which contradicts (\ref{f2}) and proves the claim (\ref{f7}).
\end{proof}

For any $z\in C\left(  \left[  0,b\right]  ,X\right)  $, the functional
$J_{\mathbb{\varepsilon},n}\left(  \cdot,z\right)  $ admits a unique minimum
$\widehat{\varphi}_{\varepsilon,n}$ that defines a map $\Phi_{\varepsilon
,n}:C\left(  \left[  0,b\right]  ,X\right)  \rightarrow X$. $\Phi
_{\varepsilon,n}$ has the following properties.

\begin{lemma}
\label{lem:2}There exists $R_{\varepsilon}>0$ such that $\left\Vert
\Phi_{\varepsilon,n}\left(  z\right)  \right\Vert <R_{\varepsilon}$ for any
$z\in B\left(  0;r\right)  ,$ $n\geq1.$
\end{lemma}

\begin{proof}
Let $z\in B\left(  0;r\right)  .$ From Lemma \ref{lem:1}(b), we see that there
exists a constant $R_{\varepsilon}>0$, such that%
\begin{equation}
\inf_{z\in B\left(  0;r\right)  }\frac{J_{\mathbb{\varepsilon},n}\left(
\varphi;z\right)  }{\left\Vert \varphi\right\Vert }\geq\frac
{\mathbb{\varepsilon}}{2},\ \ \ \left\Vert \varphi\right\Vert \geq
R_{\varepsilon}. \label{f4}%
\end{equation}
On the other hand, by the definition of $\Phi_{\varepsilon}$,%
\begin{equation}
J_{\mathbb{\varepsilon},n}\left(  \Phi_{\varepsilon,n}\left(  z\right)
;z\right)  \leq J_{\mathbb{\varepsilon},n}\left(  0;z\right)  =0. \label{f5}%
\end{equation}
Hence, combining (\ref{f4}) and (\ref{f5}), we have%
\[
\left\Vert \Phi_{\varepsilon,n}\left(  z\right)  \right\Vert <R_{\varepsilon
},\ \ \text{for\ all}\ z\in B\left(  0;r\right)  .
\]

\end{proof}

\begin{lemma}
\label{lem:3}For any $z_{l},$ $z\in B\left(  0;r\right)  $ satisfying%
\[
z_{l}\longrightarrow z,\ \ \ \text{in\ \ }C\left(  \left[  0,b\right]
,X\right)  \text{,}%
\]
it holds that%
\[
\lim_{l\rightarrow\infty}\left\Vert \Phi_{\varepsilon,n}\left(  z_{l}\right)
-\Phi_{\varepsilon,n}\left(  z\right)  \right\Vert =0.
\]

\end{lemma}

\begin{proof}
By Lemma \ref{lem:2}, we have boundedness of $\widehat{\varphi}_{\varepsilon
,n,l}=\Phi_{\varepsilon,n}\left(  z_{l}\right)  $. Consequently, we may assume
that $\widehat{\varphi}_{\varepsilon,n,l}\overset{w}{\longrightarrow
}\widetilde{\varphi}_{\varepsilon,n}$. Thus, by the definition of
$J_{\mathbb{\varepsilon},n}$ and the optimality of both $\widehat{\varphi
}_{\varepsilon,n,l}=\Phi_{\varepsilon,n}\left(  z_{l}\right)  \ $and
$\widehat{\varphi}_{\varepsilon,n}=\Phi_{\varepsilon,n}\left(  z\right)  $,
one has%
\begin{align*}
J_{\mathbb{\varepsilon},n}\left(  \widehat{\varphi}_{\varepsilon,n};z\right)
&  \leq J_{\mathbb{\varepsilon},n}\left(  \widetilde{\varphi}_{\varepsilon
,n};z\right)  \leq\underline{\lim}_{l\rightarrow\infty}J_{\mathbb{\varepsilon
},n}\left(  \widehat{\varphi}_{\varepsilon,n,l};z_{l}\right) \\
&  \leq\overline{\lim_{l\rightarrow\infty}}J_{\mathbb{\varepsilon},n}\left(
\widehat{\varphi}_{\varepsilon,n,l};z_{l}\right)  \leq\lim_{l\rightarrow
\infty}J_{\mathbb{\varepsilon},n}\left(  \widehat{\varphi}_{\varepsilon
,n};z_{l}\right)  =J_{\mathbb{\varepsilon},n}\left(  \widehat{\varphi
}_{\varepsilon,n};z\right)  .
\end{align*}
Hence, the equalities hold in the above. That means that $\widetilde{\varphi
}_{\varepsilon,n}$ is also a minimum of $J_{\mathbb{\varepsilon},n}\left(
\cdot;z\right)  $. By the uniqueness of the minimum, it is necessary that
$\widetilde{\varphi}_{\varepsilon,n}=\widehat{\varphi}_{\varepsilon,n}$.
Therefore%
\begin{align*}
\lim_{l\rightarrow\infty}J_{\mathbb{\varepsilon},n}\left(  \widehat{\varphi
}_{\varepsilon,n};z_{l}\right)   &  =J_{\mathbb{\varepsilon},n}\left(
\widehat{\varphi}_{\varepsilon,n};z\right)  ,\\
\lim_{l\rightarrow\infty}\int_{0}^{b}\left(  b-s\right)  ^{q-1}\left\Vert
B^{\ast}T_{q}^{\ast}\left(  b-s\right)  \Pi^{\ast}\widehat{\varphi
}_{\varepsilon,n,l}\right\Vert ^{2}ds  &  =\int_{0}^{b}\left(  b-s\right)
^{q-1}\left\Vert B^{\ast}T_{q}^{\ast}\left(  b-s\right)  \Pi^{\ast}%
\widehat{\varphi}_{\varepsilon,n}\right\Vert ^{2}ds,\\
\lim_{l\rightarrow\infty}\left\langle \widehat{\varphi}_{\varepsilon
,n,l},h_{n}\left(  z_{l}\right)  \right\rangle  &  =\left\langle
\widehat{\varphi}_{\varepsilon,n},h_{n}\left(  z\right)  \right\rangle
,\ \ \ \left\Vert \widehat{\varphi}_{\varepsilon,n}\right\Vert \leq
\lim_{l\rightarrow\infty}\left\Vert \widehat{\varphi}_{\varepsilon
,n,l}\right\Vert .
\end{align*}
These relations imply that%
\begin{equation}
\lim_{l\rightarrow\infty}\left\Vert \widehat{\varphi}_{\varepsilon
,n,l}\right\Vert =\left\Vert \widehat{\varphi}_{\varepsilon,n}\right\Vert .
\label{f6}%
\end{equation}
Because $X$ is Hilbert space, from $\widehat{\varphi}_{\varepsilon
,n,l}\overset{w}{\longrightarrow}\widehat{\varphi}_{\varepsilon,n}$ and
(\ref{f6}), we obtain the strong convergence of $\widehat{\varphi
}_{\varepsilon,n,l}$ to $\widehat{\varphi}_{\varepsilon,n}$.
\end{proof}

For fixed $n\geq1$, set $\Theta_{\varepsilon,n}:C\left(  \left[  0,b\right]
,X\right)  \rightarrow C\left(  \left[  0,b\right]  ,X\right)  $ defined by%
\begin{equation}
\left(  \Theta_{\varepsilon,n}z\right)  \left(  t\right)  =S_{q}\left(
t\right)  \left(  y_{0}-S\left(  \frac{1}{n}\right)  g\left(  z\right)
\right)  +\int_{0}^{t}\left(  t-s\right)  ^{q-1}T_{q}\left(  t-s\right)
\left[  Bu_{\varepsilon,n}\left(  s,z\right)  +f\left(  s,z\left(  s\right)
\right)  \right]  ds, \label{fx1}%
\end{equation}
with
\begin{equation}
u_{\varepsilon,n}\left(  s,z\right)  =B^{\ast}T_{q}^{\ast}\left(  b-s\right)
\Pi^{\ast}\widehat{\varphi}_{\varepsilon,n}=B^{\ast}T_{q}^{\ast}\left(
b-s\right)  \Pi^{\ast}\Phi_{\varepsilon,n}\left(  z\right)  . \label{c1}%
\end{equation}
We will prove $\Theta_{\varepsilon,n}$ has a fixed point by using Schauder's
fixed point theorem.

\begin{theorem}
\label{thm:app}Assume that the hypotheses of Theorem \ref{thm:main} are
satisfied. Then for $n\geq1$, the approximate control operator $\Theta
_{\varepsilon,n}$ has at least one fixed point in $C\left(  \left[
0,b\right]  ,X\right)  $.
\end{theorem}

\begin{proof}
Step 1: For any $n\geq1$, $\Theta_{\varepsilon,n}$ is continuous on $C\left(
\left[  0,b\right]  ,X\right)  $.

Let $\left\{  z_{m}:m\geq1\right\}  $ be a sequence in $C\left(  \left[
0,b\right]  ,X\right)  $ with lim$_{m\rightarrow\infty}z_{m}=z$ in $C\left(
\left[  0,b\right]  ,X\right)  $. By the continuity of $f$ and $u_{\varepsilon
,n}$, we deduce that $\left(  f\left(  s,z_{m}\left(  s\right)  \right)
,u_{\varepsilon,n}\left(  s,z_{m}\right)  \right)  $ converges to $\left(
f\left(  s,z\left(  s\right)  \right)  ,u_{\varepsilon,n}\left(  s,z\right)
\right)  $ uniformly for $s\in\left[  0,b\right]  $, and we have%
\begin{gather*}
\left\Vert \left(  \Theta_{\varepsilon,n}z_{m}\right)  \left(  t\right)
-\left(  \Theta_{\varepsilon,n}z\right)  \left(  t\right)  \right\Vert \\
\leq M_{S}^{2}\left\Vert g\left(  z_{m}\right)  -g\left(  z\right)
\right\Vert +\frac{M_{S}b^{q}}{\Gamma\left(  q+1\right)  }\left\Vert f\left(
\cdot,z_{m}\left(  \cdot\right)  \right)  -f\left(  \cdot,z\left(
\cdot\right)  \right)  \right\Vert _{C}\\
+\frac{M_{S}b^{q}}{\Gamma\left(  q+1\right)  }\left\Vert u_{\varepsilon
,n}\left(  \cdot,z_{m}\right)  -u_{\varepsilon,n}\left(  \cdot,z\right)
\right\Vert _{C}\rightarrow0,\ \ \ \text{as\ \ \ }n\rightarrow\infty,
\end{gather*}
which implies that the mapping $\Theta_{\varepsilon,n}$ is continuous on
$C\left(  \left[  0,b\right]  ,X\right)  $.

Step 2: There is a positive number $r\left(  \varepsilon\right)  >0$ such that
$\Theta_{\varepsilon,n}$ maps $B\left(  0;r\left(  \varepsilon\right)
\right)  $ into itself.

We see that%
\begin{align*}
&  \left\Vert \left(  \Theta_{\varepsilon,n}z_{k}\right)  \left(
t_{k}\right)  \right\Vert \\
&  \leq M_{S}^{2}\Lambda_{g}+\frac{M_{S}b^{q}}{\Gamma\left(  q+1\right)
}\left(  \left\Vert n\right\Vert _{C}+M_{S}M_{B}R_{\varepsilon}\right)
:=r\left(  \varepsilon\right)  .
\end{align*}

Step 3: For any $n\geq1$, $\Theta_{\varepsilon,n}$ is compact.

At the end, applying the Schauder fixed point theorem we obtain that for each
$n\geq1$, $\Theta_{\varepsilon,n}$ has at least one fixed point in $B\left(
0;r\left(  \varepsilon\right)  \right)  $.
\end{proof}

Assume that $z_{\varepsilon,n}\in B\left(  0;r\left(  \varepsilon\right)
\right)  \subset C\left(  \left[  0,b\right]  ,X\right)  $ is a fixed point of
(\ref{fx1})%
\[
\Theta_{\varepsilon,n}z_{\varepsilon,n}=z_{\varepsilon,n}%
\]
and $\Phi_{\varepsilon,n}\left(  z_{\varepsilon,n}\right)  $ is minimizer of
$J_{\mathbb{\varepsilon},n}\left(  \varphi;z_{\varepsilon,n}\right)  $ and
\[
u_{\varepsilon,n}\left(  s,z_{\varepsilon,n}\right)  =B^{\ast}T_{q}^{\ast
}\left(  b-s\right)  \Pi^{\ast}\Phi_{\varepsilon,n}\left(  z_{\varepsilon
,n}\right)  ,
\]
is the corresponding control. Moreover, assume that
\[
z_{\varepsilon,n}\overset{n\rightarrow\infty}{\rightarrow}z_{\varepsilon
}\ \text{strongly}\ \text{in }C\left(  \left[  0,b\right]  ,X\right)
,\ \text{as\ }n\rightarrow\infty.
\]
$\Phi_{\varepsilon}\left(  z_{\varepsilon}\right)  $ is minimizer of
$J_{\mathbb{\varepsilon}}\left(  \varphi;z_{\varepsilon}\right)  $ and
\[
u_{\varepsilon}\left(  s,z_{\varepsilon}\right)  =B^{\ast}T_{q}^{\ast}\left(
b-s\right)  \Pi^{\ast}\Phi_{\varepsilon}\left(  z_{\varepsilon}\right)  ,
\]
is the corresponding control.

\begin{lemma}
\label{lem:5}Assume that%
\[
\lim_{n\rightarrow\infty}\left\Vert z_{\varepsilon,n}-z_{\varepsilon
}\right\Vert _{C}=0.
\]
Then
\begin{align*}
\Phi_{\varepsilon,n}\left(  z_{\varepsilon,n}\right)   &  \rightharpoonup
\Phi_{\varepsilon}\left(  z_{\varepsilon}\right)  \ \ \text{weakly in }X,\\
\lim_{n\rightarrow\infty}\left\Vert u_{\varepsilon,n}\left(  s,z_{\varepsilon
,n}\right)  -u_{\varepsilon}\left(  s,z_{\varepsilon}\right)  \right\Vert
_{C}  &  =0.
\end{align*}

\end{lemma}

\begin{proof}
By definition of the minimizing functional $\Phi_{\varepsilon,n}\left(
z_{\varepsilon,n}\right)  $ and $\Phi_{\varepsilon}\left(  z_{\varepsilon
}\right)  $ are minimizers of
\begin{align*}
J_{\mathbb{\varepsilon},n}\left(  \varphi;z_{\varepsilon,n}\right)   &
=\frac{1}{2}\int_{0}^{b}\left(  b-s\right)  ^{q-1}\left\Vert B^{\ast}%
T_{q}^{\ast}\left(  b-s\right)  \Pi^{\ast}\varphi\right\Vert ^{2}%
ds+\mathbb{\varepsilon}\left\Vert \varphi\right\Vert -\left\langle
\varphi,h_{n}\left(  z_{\varepsilon,n}\right)  \right\rangle ,\\
J_{\mathbb{\varepsilon}}\left(  \varphi;z_{\varepsilon}\right)   &  =\frac
{1}{2}\int_{0}^{b}\left(  b-s\right)  ^{q-1}\left\Vert B^{\ast}T_{q}^{\ast
}\left(  b-s\right)  \Pi^{\ast}\varphi\right\Vert ^{2}ds+\mathbb{\varepsilon
}\left\Vert \varphi\right\Vert -\left\langle \varphi,h\left(  z_{\varepsilon
}\right)  \right\rangle ,\
\end{align*}
correspondingly. By Lemma \ref{lem:2}, we have boundedness of $\Phi
_{\varepsilon,n}\left(  z_{\varepsilon,n}\right)  $. Consequently, we may
assume that $\Phi_{\varepsilon,n}\left(  z_{\varepsilon,n}\right)
\rightharpoonup\widetilde{\Phi}_{\varepsilon}$ weakly in $X$. Thus, by the
definition of $J_{\mathbb{\varepsilon},n}$ and the optimality of both
$\Phi_{\varepsilon,n}\left(  z_{\varepsilon,n}\right)  $ and $\Phi
_{\varepsilon}\left(  z_{\varepsilon}\right)  $, one has%
\begin{equation}
J_{\mathbb{\varepsilon}}\left(  \Phi_{\varepsilon}\left(  z_{\varepsilon
}\right)  ;z_{\varepsilon}\right)  \leq J_{\mathbb{\varepsilon}}\left(
\widetilde{\Phi}_{\varepsilon};z_{\varepsilon}\right)  \leq\underline{\lim
}_{n\rightarrow\infty}J_{\mathbb{\varepsilon}}\left(  \Phi_{\varepsilon
,n}\left(  z_{\varepsilon,n}\right)  ;z_{\varepsilon}\right)  , \label{qq1}%
\end{equation}
and%
\begin{equation}
\underline{\lim}_{n\rightarrow\infty}J_{\mathbb{\varepsilon},n}\left(
\Phi_{\varepsilon,n}\left(  z_{\varepsilon,n}\right)  ;z_{\varepsilon
,n}\right)  \leq\overline{\lim_{n\rightarrow\infty}}J_{\mathbb{\varepsilon}%
,n}\left(  \Phi_{\varepsilon,n}\left(  z_{\varepsilon,n}\right)
;z_{\varepsilon,n}\right)  \leq\lim_{n\rightarrow\infty}J_{\mathbb{\varepsilon
},n}\left(  \Phi_{\varepsilon}\left(  z_{\varepsilon}\right)  ;z_{\varepsilon
,n}\right)  =J_{\mathbb{\varepsilon}}\left(  \Phi_{\varepsilon}\left(
z_{\varepsilon}\right)  ;z_{\varepsilon}\right)  . \label{qq2}%
\end{equation}
Noting that $\underline{\lim}_{n\rightarrow\infty}J_{\mathbb{\varepsilon}%
}\left(  \Phi_{\varepsilon,n}\left(  z_{\varepsilon,n}\right)  ;z_{\varepsilon
}\right)  =\underline{\lim}_{n\rightarrow\infty}J_{\mathbb{\varepsilon}%
,n}\left(  \Phi_{\varepsilon,n}\left(  z_{\varepsilon,n}\right)
;z_{\varepsilon,n}\right)  $ and combining (\ref{qq1}) and (\ref{qq2}) we have%
\begin{align*}
J_{\mathbb{\varepsilon}}\left(  \Phi_{\varepsilon}\left(  z_{\varepsilon
}\right)  ;z_{\varepsilon}\right)   &  =J_{\mathbb{\varepsilon}}\left(
\widetilde{\Phi}_{\varepsilon};z_{\varepsilon}\right)  ,\\
\underline{\lim}_{n\rightarrow\infty}J_{\mathbb{\varepsilon},n}\left(
\Phi_{\varepsilon,n}\left(  z_{\varepsilon,n}\right)  ;z_{\varepsilon
,n}\right)   &  =\overline{\lim_{n\rightarrow\infty}}J_{\mathbb{\varepsilon
},n}\left(  \Phi_{\varepsilon,n}\left(  z_{\varepsilon,n}\right)
;z_{\varepsilon,n}\right)  =J_{\mathbb{\varepsilon}}\left(  \Phi_{\varepsilon
}\left(  z_{\varepsilon}\right)  ;z_{\varepsilon}\right)  .
\end{align*}
That means that $\widetilde{\Phi}_{\varepsilon}$ is also a minimum of
$J_{\mathbb{\varepsilon}}\left(  \cdot;z_{\varepsilon}\right)  $. By the
uniqueness of the minimum, it is necessary that $\widetilde{\Phi}%
_{\varepsilon}=\Phi_{\varepsilon}\left(  z_{\varepsilon}\right)  $. Therefore
\begin{align*}
\lim_{n\rightarrow\infty}J_{\mathbb{\varepsilon},n}\left(  \Phi_{\varepsilon
,n}\left(  z_{\varepsilon,n}\right)  ;z_{\varepsilon,n}\right)   &
=J_{\mathbb{\varepsilon}}\left(  \Phi_{\varepsilon}\left(  z_{\varepsilon
}\right)  ;z_{\varepsilon}\right)  ,\\
\lim_{n\rightarrow\infty}\int_{0}^{b}\left(  b-s\right)  ^{q-1}\left\Vert
B^{\ast}T_{q}^{\ast}\left(  b-s\right)  \Pi^{\ast}\Phi_{\varepsilon,n}\left(
z_{\varepsilon,n}\right)  \right\Vert ^{2}ds  &  =\int_{0}^{b}\left(
b-s\right)  ^{q-1}\left\Vert B^{\ast}T_{q}^{\ast}\left(  b-s\right)  \Pi
^{\ast}\Phi_{\varepsilon}\left(  z_{\varepsilon}\right)  \right\Vert ^{2}ds,\\
\lim_{n\rightarrow\infty}\left\langle \Phi_{\varepsilon,n}\left(
z_{\varepsilon,n}\right)  ,h_{n}\left(  z_{\varepsilon,n}\right)
\right\rangle  &  =\left\langle \Phi_{\varepsilon}\left(  z_{\varepsilon
}\right)  ,h\left(  z_{\varepsilon}\right)  \right\rangle ,\ \ \ \left\Vert
\Phi_{\varepsilon}\left(  z_{\varepsilon}\right)  \right\Vert \leq
\lim_{n\rightarrow\infty}\left\Vert \Phi_{\varepsilon,n}\left(  z_{\varepsilon
,n}\right)  \right\Vert .
\end{align*}
These relations imply that%
\begin{equation}
\lim_{n\rightarrow\infty}\left\Vert \Phi_{\varepsilon,n}\left(  z_{\varepsilon
,n}\right)  \right\Vert =\left\Vert \Phi_{\varepsilon}\left(  z_{\varepsilon
}\right)  \right\Vert . \label{ff6}%
\end{equation}
As $X$ is Hilbert space, from $\Phi_{\varepsilon,n}\left(  z_{\varepsilon
,n}\right)  \rightharpoonup\Phi_{\varepsilon}\left(  z_{\varepsilon}\right)  $
weakly in $X$ and (\ref{ff6}), we obtain the strong convergence of
$\Phi_{\varepsilon,n}\left(  z_{\varepsilon,n}\right)  $ to $\Phi
_{\varepsilon}\left(  z_{\varepsilon}\right)  $.
\end{proof}

\section{Proof of the main result}

Let $y_{0}\in X,\ y_{b}\in E$ be any given two points and $\mathbb{\varepsilon
}>0$ be any given accuracy. Then it is seen that for all $t\in\left[
0,b\right]  $%
\begin{align*}
\left(  \Theta_{\varepsilon}z\right)  \left(  t\right)   &  =S_{q}\left(
t\right)  \left(  y_{0}-g\left(  z\right)  \right)  +\int_{0}^{t}\left(
t-s\right)  ^{q-1}T_{q}\left(  t-s\right)  \left[  Bu_{\varepsilon}\left(
s,z\right)  +f\left(  s,z\left(  s\right)  \right)  \right]  ds,\\
u\left(  s,z\right)   &  =B^{\ast}T_{q}^{\ast}\left(  b-s\right)  \Pi^{\ast
}\widehat{\varphi}_{\varepsilon}=B^{\ast}T_{q}^{\ast}\left(  b-s\right)
\Pi^{\ast}\Phi_{\varepsilon}\left(  z\right)  .
\end{align*}

\begin{theorem}
\label{thm:2}Assume that the hypotheses of Theorem \ref{thm:main} are
satisfied. Then the control operator $\Theta_{\varepsilon}$ has at least one
fixed point in $C\left(  \left[  0,b\right]  ,X\right)  $.
\end{theorem}

\begin{proof}
Let $n\geq1$ be fixed. By Theorem \ref{fx1} the approximate operator
$\Theta_{\varepsilon,n}$ defined by (\ref{fx1}) has a fixed point, say
$y_{\varepsilon,n}$. We define the approximate solution set $D$ by%
\[
D=\left\{  y_{\varepsilon,n}\in C\left(  \left[  0,b\right]  ,X\right)
:\Theta_{\varepsilon,n}y_{\varepsilon,n}=y_{\varepsilon,n},\ n\geq1\right\}
.
\]

Step 1: $D\left(  0\right)  $ is relatively compact in $X$.

For $y_{\varepsilon,n}\in D$, $n\geq1$, define%
\[
\widetilde{y}_{\varepsilon,n}\left(  t\right)  =\left\{
\begin{tabular}
[c]{lll}%
$y_{\varepsilon,n}\left(  t\right)  ,$ &  & $\delta\leq t\leq b,$\\
$y_{\varepsilon,n}\left(  \delta\right)  $ &  & $0\leq t\leq\delta,$%
\end{tabular}
\ \right.
\]
where $\delta$ comes from the condition (G). It is easily seen that $\left\{
\widetilde{y}_{\varepsilon,n}:n\geq1\right\}  $ is relatively compact in
$C\left(  \left[  0,b\right]  ,X\right)  $. Without loss of generality, we may
suppose that $\widetilde{y}_{\varepsilon,n}\rightarrow\widetilde
{y}_{\varepsilon}\in C\left(  \left[  0,b\right]  ,X\right)  $ as
$n\rightarrow\infty$. By assumption (G), we get that $g\left(  y_{\varepsilon
,n}\right)  =g\left(  \widetilde{y}_{\varepsilon,n}\right)  \rightarrow
g\left(  \widetilde{y}_{\varepsilon}\right)  $. Thus by the continuity of
$S_{q}\left(  t\right)  $ and $g$, we have%
\begin{gather*}
\left\Vert y_{\varepsilon,n}\left(  0\right)  -\left(  y_{0}-g\left(
\widetilde{y}_{\varepsilon}\right)  \right)  \right\Vert =\left\Vert S\left(
\frac{1}{n}\right)  g\left(  y_{\varepsilon,n}\right)  -g\left(  \widetilde
{y}_{\varepsilon}\right)  \right\Vert \\
\leq\left\Vert S\left(  \frac{1}{n}\right)  g\left(  y_{\varepsilon,n}\right)
-S\left(  \frac{1}{n}\right)  g\left(  \widetilde{y}_{\varepsilon}\right)
\right\Vert +\left\Vert S\left(  \frac{1}{n}\right)  g\left(  \widetilde
{y}_{\varepsilon}\right)  -g\left(  \widetilde{y}_{\varepsilon}\right)
\right\Vert \\
\leq M_{S}\left\Vert g\left(  y_{\varepsilon,n}\right)  -g\left(
\widetilde{y}_{\varepsilon}\right)  \right\Vert +\left\Vert S\left(  \frac
{1}{n}\right)  g\left(  \widetilde{y}_{\varepsilon}\right)  -g\left(
\widetilde{y}_{\varepsilon}\right)  \right\Vert \rightarrow0,
\end{gather*}
as $n\rightarrow\infty$. So, $D\left(  0\right)  =\left\{  y_{\varepsilon
,n}\left(  0\right)  =y_{0}-S\left(  \frac{1}{n}\right)  g\left(
y_{\varepsilon,n}\right)  \right\}  $ is relatively compact in $X$.

Step 2: For each $t\in\left(  0,b\right]  $ the set $D\left(  t\right)
:=\left\{  y_{\varepsilon,n}\left(  t\right)  :n\geq1\right\}  $ is relatively
compact in $X$.

Step 3: $D$ is equicontinuous at $t=0.$

First of all note that%
\begin{align*}
&  \sup\left\{  \left\Vert Bu_{\varepsilon,n}\left(  s,y_{\varepsilon
,n}\right)  +f\left(  s,y_{\varepsilon,n}\left(  s\right)  \right)
\right\Vert :s\in\left[  0,b\right]  ,y_{\varepsilon,n}\in D\right\} \\
&  \leq M_{B}^{2}M_{S}R_{\varepsilon}+\left\Vert n\right\Vert _{C}.
\end{align*}
For $0<t<b$, we have%
\begin{gather*}
\left\Vert y_{\varepsilon,n}\left(  t\right)  -y_{\varepsilon,n}\left(
0\right)  \right\Vert =\left\Vert S_{q}\left(  t\right)  S\left(  \frac{1}%
{n}\right)  g\left(  y_{\varepsilon,n}\right)  -S\left(  \frac{1}{n}\right)
g\left(  y_{\varepsilon,n}\right)  \right\Vert \\
+\frac{M_{S}}{\Gamma\left(  q\right)  }\int_{0}^{t}\left(  t-s\right)
^{q-1}ds\sup\left\{  \left\Vert Bu_{\varepsilon,n}\left(  s,y_{\varepsilon
,n}\right)  +f\left(  s,y_{\varepsilon,n}\left(  s\right)  \right)
\right\Vert :s\in\left[  0,b\right]  ,y_{\varepsilon,n}\in D\right\} \\
\leq\left\Vert \left(  S_{q}\left(  t\right)  -I\right)  S\left(  \frac{1}%
{n}\right)  g\left(  y_{\varepsilon,n}\right)  \right\Vert \\
+\frac{M_{S}t^{q}}{\Gamma\left(  q+1\right)  }\left(  M_{B}^{2}M_{S}%
R_{\varepsilon}+\left\Vert n\right\Vert _{C}\right)
\end{gather*}
approaches zero uniformly as $t\rightarrow0$, since $\left\{  S\left(
\frac{1}{n}\right)  g\left(  y_{\varepsilon,n}\right)  :n\geq1\right\}  $ is
relatively compact in $X$. Then we obtain the set $D$ is equicontinuous at
$t=0$.

Step 4: $D$ is equicontinuous on $\left(  0,b\right]  $.

Let $0<t_{1}<t_{2}\leq b$ and choose $\eta>0$ such that $t_{1}-\eta>0$. Then,
for any $y_{\varepsilon,n}\in D$, we have%
\begin{align*}
&  \left\Vert y_{\varepsilon,n}\left(  t_{2}\right)  -y_{\varepsilon,n}\left(
t_{1}\right)  \right\Vert \\
&  \leq\left\Vert \left(  S_{q}\left(  t_{2}\right)  -S_{q}\left(
t_{1}\right)  \right)  \left(  y_{0}-S\left(  \frac{1}{n}\right)  g\left(
y_{\varepsilon,n}\right)  \right)  \right\Vert \\
&  +\left\Vert \int_{t_{1}}^{t_{2}}\left(  t_{2}-s\right)  ^{q-1}T_{q}\left(
t_{2}-s\right)  \left(  Bu_{\varepsilon,n}\left(  s,y_{\varepsilon,n}\right)
+f\left(  s,y_{\varepsilon,n}\left(  s\right)  \right)  \right)  ds\right\Vert
\\
&  +\left\Vert \int_{0}^{t_{1}}\left(  \left(  t_{2}-s\right)  ^{q-1}-\left(
t_{1}-s\right)  ^{q-1}\right)  T_{q}\left(  t_{2}-s\right)  \left(
Bu_{\varepsilon,n}\left(  s,y_{\varepsilon,n}\right)  +f\left(
s,y_{\varepsilon,n}\left(  s\right)  \right)  \right)  ds\right\Vert \\
&  +\left\Vert \int_{0}^{t_{1}}\left(  t_{1}-s\right)  ^{q-1}\left(
T_{q}\left(  t_{2}-s\right)  -T_{q}\left(  t_{1}-s\right)  \right)  \left(
Bu_{\varepsilon,n}\left(  s,y_{\varepsilon,n}\right)  +f\left(
s,y_{\varepsilon,n}\left(  s\right)  \right)  \right)  ds\right\Vert \\
&  =:I_{1}+I_{2}+I_{3}+I_{4}.
\end{align*}

From assumption (G) and norm continuity of $S_{q}\left(  t\right)  $, $t>0$,
it follows that%
\begin{equation}
I_{1}=\left\Vert \left(  S_{q}\left(  t_{2}\right)  -S_{q}\left(
t_{1}\right)  \right)  \left(  y_{0}-S\left(  \frac{1}{n}\right)  g\left(
y_{\varepsilon,n}\right)  \right)  \right\Vert \rightarrow0, \label{q1}%
\end{equation}
as $t_{2}-t_{1}\rightarrow0,$ uniformly for all $y_{\varepsilon,n}\in D$. By
direct calculation to $I_{2}$, we obtain that%
\begin{equation}
I_{2}\leq\frac{M_{S}\left(  t_{2}-t_{1}\right)  ^{q}}{\Gamma\left(
q+1\right)  }\left(  M_{B}^{2}M_{S}R_{\varepsilon}+\left\Vert n\right\Vert
_{C}\right)  \label{q2}%
\end{equation}

As $t_{1}>0,$we obtain%
\begin{equation}
I_{3}\leq\frac{M_{S}}{\Gamma\left(  q\right)  }\int_{0}^{t_{1}}\left(  \left(
t_{2}-s\right)  ^{q-1}-\left(  t_{1}-s\right)  ^{q-1}\right)  ds\left(
M_{B}^{2}M_{S}R_{\varepsilon}+\left\Vert n\right\Vert _{C}\right)  \label{q3}%
\end{equation}

Next, for $t_{1}-\eta>0$, we have%
\begin{align}
I_{4}  &  \leq\left\Vert \left(  \int_{0}^{t_{1}-\eta}+\int_{t_{1}-\eta
}^{t_{1}}\right)  \left(  t_{1}-s\right)  ^{q-1}\left(  T_{q}\left(
t_{2}-s\right)  -T_{q}\left(  t_{1}-s\right)  \right)  \left(  Bu_{\varepsilon
,n}\left(  s,y_{\varepsilon,n}\right)  +f\left(  s,y_{\varepsilon,n}\left(
s\right)  \right)  \right)  ds\right\Vert \nonumber\\
&  \leq\frac{t_{1}^{q}-\eta^{q}}{q}\sup\left\{  \left\Vert Bu_{\varepsilon
,n}\left(  s,y_{\varepsilon,n}\right)  +f\left(  s,y_{\varepsilon,n}\left(
s\right)  \right)  \right\Vert :s\in\left[  0,b\right]  ,y_{\varepsilon,n}\in
D\right\}  \sup_{0\leq s\leq t-\eta}\left\Vert T_{q}\left(  t_{2}-s\right)
-T_{q}\left(  t_{1}-s\right)  \right\Vert \nonumber\\
&  +\frac{2M_{S}\eta^{q}}{\Gamma\left(  q+1\right)  }\left(  M_{B}^{2}%
M_{S}R_{\varepsilon}+\left\Vert n\right\Vert _{C}\right)  \label{q4}%
\end{align}

Thus, combining the above inequalities (\ref{q1})--(\ref{q4}) with the
continuity of $T_{q}\left(  t\right)  ,t>0$, in the uniform operator topology,
we obtain the equicontinuity of the set $D$ on $\left(  0,b\right]  $.

Therefore, the set $D$ is relatively compact in $C\left(  \left[  0,b\right]
,X\right)  $ and we may assume $y_{\varepsilon,n}\rightarrow y_{\varepsilon}$
for some $y_{\varepsilon}\in C\left(  \left[  0,b\right]  ,X\right)  $ as
$n\rightarrow\infty$. On the other hand, by Lemma \ref{lem:5} $u_{\varepsilon
,n}\left(  s,y_{\varepsilon,n}\right)  \rightarrow u_{\varepsilon}\left(
s,y_{\varepsilon}\right)  $ in $C\left(  \left[  0,b\right]  ,X\right)  $ as
$n\rightarrow\infty.$ By taking the limit as $n\rightarrow\infty$ in
$\Theta_{\varepsilon,n}y_{\varepsilon,n}=y_{\varepsilon,n}$ and using the
Lebesgue dominated convergence theorem, we obtain that
\[
y_{\varepsilon}\left(  t\right)  =S_{q}\left(  t\right)  \left(
y_{0}-g\left(  y_{\varepsilon}\right)  \right)  +\int_{0}^{t}\left(
t-s\right)  ^{q-1}T_{q}\left(  t-s\right)  \left[  Bu_{\varepsilon}\left(
s,y_{\varepsilon}\right)  +f\left(  s,y_{\varepsilon}\left(  s\right)
\right)  \right]  ds,
\]
for $t\in\left[  0,b\right]  $, which implies that $y_{\varepsilon}$ is a mild
solution of semilinear fractional control system (\ref{cp1}). This completes
the proof.
\end{proof}

In view of Theorem \ref{thm:2} for any $\mathbb{\varepsilon}>0$ there exists
$y_{\mathbb{\varepsilon}}\in C\left(  \left[  0,b\right]  ,X\right)  $ such
that%
\[
y_{\mathbb{\varepsilon}}\left(  t\right)  =S_{q}\left(  t\right)  \left(
y_{0}-g\left(  y_{\mathbb{\varepsilon}}\right)  \right)  +\int_{0}^{t}\left(
t-s\right)  ^{q-1}T_{q}\left(  t-s\right)  \left[  Bu_{\varepsilon}\left(
s,y_{\varepsilon}\right)  +f\left(  s,y_{\varepsilon}\left(  s\right)
\right)  \right]  ds,
\]
where $u\left(  s,y_{\mathbb{\varepsilon}}\right)  =B^{\ast}S^{\ast}\left(
b-s\right)  \Phi_{\varepsilon}\left(  y_{\mathbb{\varepsilon}}\right)  .$

Now we prove our main result.

\begin{proof}
[\textbf{Proof of Theorem \ref{thm:main}}]By Lemma \ref{lem:1}(a) we know that
$J_{\mathbb{\varepsilon}}$ is strictly convex. Then $J_{\mathbb{\varepsilon}%
}\left(  \varphi;y_{\mathbb{\varepsilon}}\right)  $ has a unique critical
point which is its minimizer:%
\[
\widehat{\varphi}_{\mathbb{\varepsilon}}\in X:J_{\mathbb{\varepsilon}}\left(
\widehat{\varphi}_{\mathbb{\varepsilon}};y_{\mathbb{\varepsilon}}\right)
=\min_{\varphi\in X}J_{\mathbb{\varepsilon}}\left(  \varphi
;y_{\mathbb{\varepsilon}}\right)  .
\]
Given any $\psi\in X$ and $\lambda\in R$ we have%
\[
J_{\mathbb{\varepsilon}}\left(  \widehat{\varphi}_{\mathbb{\varepsilon}%
};y_{\mathbb{\varepsilon}}\right)  \leq J_{\mathbb{\varepsilon}}\left(
\widehat{\varphi}_{\mathbb{\varepsilon}}+\lambda\psi;y_{\mathbb{\varepsilon}%
}\right)
\]
or, in other words,%
\begin{align*}
&  \mathbb{\varepsilon}\left\Vert \widehat{\varphi}_{\mathbb{\varepsilon}%
}\right\Vert \\
&  \leq\frac{\lambda^{2}}{2}\int_{0}^{b}\left(  b-s\right)  ^{q-1}\left\Vert
B^{\ast}T_{q}^{\ast}\left(  b-s\right)  \Pi^{\ast}\psi\right\Vert
^{2}ds+\lambda\int_{0}^{b}\left(  b-s\right)  ^{q-1}\left\langle B^{\ast}%
T_{q}^{\ast}\left(  b-s\right)  \Pi^{\ast}\widehat{\varphi}%
_{\mathbb{\varepsilon}},B^{\ast}T_{q}^{\ast}\left(  b-s\right)  \Pi^{\ast}%
\psi\right\rangle ds\\
&  +\mathbb{\varepsilon}\left\Vert \widehat{\varphi}_{\mathbb{\varepsilon}%
}+\lambda\psi\right\Vert -\lambda\left\langle \psi,h\left(
y_{\mathbb{\varepsilon}}\right)  \right\rangle .
\end{align*}
Dividing this inequality by $\lambda>0$ and letting $\lambda\longrightarrow
0^{+}$ we obtain that%
\begin{align*}
\left\langle \psi,h\left(  y_{\mathbb{\varepsilon}}\right)  \right\rangle  &
\leq\int_{0}^{b}\left(  b-s\right)  ^{q-1}\left\langle B^{\ast}T_{q}^{\ast
}\left(  b-s\right)  \Pi^{\ast}\widehat{\varphi}_{\mathbb{\varepsilon}%
},B^{\ast}T_{q}^{\ast}\left(  b-s\right)  \Pi^{\ast}\psi\right\rangle ds\\
&  +\mathbb{\varepsilon}\lim\inf_{\lambda\longrightarrow0^{+}}\frac{\left\Vert
\widehat{\varphi}_{\mathbb{\varepsilon}}+\lambda\psi\right\Vert -\left\Vert
\widehat{\varphi}_{\mathbb{\varepsilon}}\right\Vert }{\lambda}\\
&  \leq\int_{0}^{b}\left(  b-s\right)  ^{q-1}\left\langle B^{\ast}T_{q}^{\ast
}\left(  b-s\right)  \Pi^{\ast}\widehat{\varphi}_{\mathbb{\varepsilon}%
},B^{\ast}T_{q}^{\ast}\left(  b-s\right)  \Pi^{\ast}\psi\right\rangle
ds+\mathbb{\varepsilon}\left\Vert \psi\right\Vert .
\end{align*}
Repeating this argument with $\lambda<0$ we obtain finally that
\begin{equation}
\left\vert \int_{0}^{b}\left(  b-s\right)  ^{q-1}\left\langle B^{\ast}%
T_{q}^{\ast}\left(  b-s\right)  \Pi^{\ast}\widehat{\varphi}%
_{\mathbb{\varepsilon}},B^{\ast}T_{q}^{\ast}\left(  b-s\right)  \Pi^{\ast}%
\psi\right\rangle ds-\left\langle \psi,h\left(  y_{\mathbb{\varepsilon}%
}\right)  \right\rangle \right\vert \leq\mathbb{\varepsilon}\left\Vert
\psi\right\Vert . \label{f9}%
\end{equation}
On the other hand, with $u_{\varepsilon}=B^{\ast}T_{q}^{\ast}\left(
b-s\right)  \Pi^{\ast}\widehat{\varphi}_{\mathbb{\varepsilon}}$ we have%
\begin{align}
&  \int_{0}^{b}\left(  b-s\right)  ^{q-1}\left\langle \Pi T_{q}\left(
b-s\right)  BB^{\ast}T_{q}^{\ast}\left(  b-s\right)  \Pi^{\ast}\widehat
{\varphi}_{\mathbb{\varepsilon}}ds-h\left(  y_{\mathbb{\varepsilon}}\right)
,\psi\right\rangle \nonumber\\
&  =\left\langle \Pi y_{\mathbb{\varepsilon}}\left(  b\right)  -y_{b}%
,\psi\right\rangle ,\nonumber\\
h\left(  y_{\mathbb{\varepsilon}}\right)   &  =y_{b}-\Pi S_{q}\left(
b\right)  \left(  y_{0}-g\left(  y_{\mathbb{\varepsilon}}\right)  \right)
-\int_{0}^{b}\left(  b-s\right)  ^{q-1}\Pi T_{q}\left(  b-s\right)  f\left(
s,y_{\mathbb{\varepsilon}}\left(  s\right)  \right)  ds. \label{f8}%
\end{align}
Then, combining (\ref{f9}) and (\ref{f8}) we obtain that%
\[
\left\vert \left\langle \Pi y_{\mathbb{\varepsilon}}\left(  b\right)
-y_{b},\psi\right\rangle \right\vert \leq\mathbb{\varepsilon}\left\Vert
\psi\right\Vert
\]
holds for any $\psi\in X$. Thus%
\[
\left\Vert \Pi y_{\mathbb{\varepsilon}}\left(  b\right)  -y_{b}\right\Vert
\leq\mathbb{\varepsilon}.
\]

\end{proof}

\section{Applications}

\textbf{Example 1: }Consider the following initial-boundary value problem of
fractional parabolic control system with Caputo fractional derivatives, based
on example 2 in \cite{dm1}:%
\begin{equation}
\left\{
\begin{tabular}
[c]{lll}%
$^{C}D_{0,t}^{2/3}x\left(  t,\theta\right)  =\dfrac{\partial^{2}}%
{\partial\theta^{2}}x\left(  t,\theta\right)  +Bu\left(  t,\theta\right)
+f\left(  t,x\left(  t,\theta\right)  \right)  ,$ & $t\in\left[  0,b\right]
,$ & $\theta\in\left[  0,\pi\right]  ,$\\
$x\left(  t,0\right)  =x\left(  t,\pi\right)  =0,$ & $t\in\left[  0,b\right]
,$ & \\
$x\left(  0,\theta\right)  =x_{0}\left(  \theta\right)  -\sum_{i=0}^{m}%
\int_{0}^{\pi}k\left(  \theta,s\right)  x\left(  t_{i},s\right)  ds,$ &
$t\in\left[  0,b\right]  ,$ & $\theta\in\left[  0,\pi\right]  .$%
\end{tabular}
\ \ \ \ \ \ \ \ \ \ \ \right.  \label{ex1}%
\end{equation}
Here, $f$ is a given function, $m$ is a positive integer, $0<t_{0}%
<t_{1}<...<t_{m}<b,$ $k\left(  \cdot,\cdot\right)  \in L^{2}\left(  \left[
0,\pi\right]  \times\left[  0,\pi\right]  ,R^{+}\right)  .\ $

Take $X=L^{2}\left[  0,\pi\right]  $ and the operator $A:D\left(  A\right)
\subset X\rightarrow X$ is defined by%
\[
Ay=y^{\prime\prime},
\]
where the domain $D\left(  A\right)  $ is defined by%
\[
\left\{  x\in X:x,\ x^{\prime}\ \text{are absolutely continuous,\ \ }%
x^{\prime\prime}\in X,\ \ x\left(  0\right)  =x\left(  \pi\right)  =0\right\}
.
\]
Then, $A$ can be written as%
\[
Ay=-\sum_{n=1}^{\infty}n^{2}\left\langle y,e_{n}\right\rangle e_{n},\ \ \ y\in
D\left(  A\right)  ,
\]
\bigskip where $e_{n}\left(  x\right)  =\sqrt{\frac{2}{\pi}}\sin nx$ is an
orthonormal basis of $X$. It is well known that $A$ is the infinitesimal
generator of a compact, analytic and self-adjoint semigroup $S\left(
t\right)  ,$ $t>0$, in $X$ given by%
\[
S\left(  t\right)  y=\sum_{n=1}^{\infty}\exp\left(  -n^{2}t\right)
\left\langle y,e_{n}\right\rangle e_{n},\ \ \ y\in X.
\]
Now, define an infinite-dimensional space $U$ by%
\[
U=\left\{  u=\sum_{n=2}^{\infty}u_{n}e_{n}\left(  \theta\right)  :\sum
_{n=2}^{\infty}u_{n}^{2}<\infty\right\}  ,
\]
with the norm $\left\Vert u\right\Vert ^{2}=\sum_{n=2}^{\infty}u_{n}^{2}$.
Define the operator $B:U\rightarrow X$ as follows%
\[
Bu=2u_{2}e_{1}\left(  \theta\right)  +\sum_{n=2}^{\infty}u_{n}e_{n}\left(
\theta\right)  .
\]

The system (\ref{ex1}) can be reformulated as the following nonlocal
controllability problem in $X.$%
\begin{align}
^{C}D_{0,t}^{2/3}y\left(  t\right)   &  =Ay\left(  t\right)  +Bu\left(
t\right)  +f\left(  t,y\left(  t\right)  \right)  ,\nonumber\\
y\left(  0\right)   &  =y_{0}-g\left(  y\right)  ,
\end{align}
where $y\left(  t\right)  =x\left(  t,\cdot\right)  ,\ f:\left[  0,b\right]
\times X\rightarrow X$ is given by $f\left(  t,y\left(  t\right)  \right)
=f\left(  t,x\left(  t,\cdot\right)  \right)  $ and the function $g:C\left(
\left[  0,b\right]  ,X\right)  \rightarrow X$ is given by%
\[
g\left(  y\right)  =\sum_{i=0}^{m}\int_{0}^{\pi}k\left(  \cdot,s\right)
x\left(  t_{i},s\right)  ds
\]

This implies that the condition (AC) is satisfied. If we assume that $f$ and
$g$ satisfy the assumptions (F) and (G), then by Theorem \ref{thm:main}, we
obtain that the system (\ref{ex1}) is approximately controllable on $\left[
0,T\right]  .$

\textbf{Example 2: }Consider the following affine hereditary differential
system in the Hilbert space $E$
\begin{align}
y^{\prime}\left(  t\right)   &  =Ay\left(  t\right)  +Ny\left(  t-h\right)
+\int_{-h}^{0}M\left(  \theta\right)  y\left(  t+\theta\right)  d\theta
+Bu\left(  t\right)  +f\left(  t,y\left(  t\right)  \right)  ,\label{ss1}\\
y\left(  0\right)   &  =\xi+\int_{\delta}^{b}h\left(  s,y\left(  s\right)
\right)  ds,\ \ y\left(  \theta\right)  =\eta\left(  \theta\right)
,\ \ \ -h\leq\theta<0,\ h>0,\ \delta>0,\nonumber
\end{align}
where $A$ is the generator of the strongly continuous compact semigroup
$S\left(  t\right)  :E\rightarrow E,\ t>0,$ $N\in\mathcal{L}\left(  E\right)
$, $M$ is a Lebesgue measurable and essentially bounded $\mathcal{L}\left(
E\right)  $-valued function on $\left[  -h,0\right]  $, $B\in\mathcal{L}%
\left(  U,E\right)  $, $\xi\in E$, $\eta\in L^{2}\left(  \left[  -h,0\right]
,E\right)  $. Introduce $M^{2}\left(  \left[  -h,0\right]  ,E\right)
:=E\times L^{2}\left(  \left[  -h,0\right]  ,E\right)  $ and $\widetilde
{A}:D\left(  \widetilde{A}\right)  \subset M^{2}\left(  \left[  -h,0\right]
,E\right)  \rightarrow M^{2}\left(  \left[  -h,0\right]  ,E\right)  $ defined
by%
\[
\widetilde{A}\left[
\begin{array}
[c]{c}%
\xi\\
\eta
\end{array}
\right]  =\left[
\begin{array}
[c]{c}%
A\xi+N\eta\left(  -h\right)  +\int_{-h}^{0}M\left(  \theta\right)  \eta\left(
\theta\right)  d\theta\\
\frac{d}{d\theta}\left(  \eta-\xi\right)
\end{array}
\right]  ,
\]
where%
\[
D\left(  \widetilde{A}\right)  =\left\{  \left[
\begin{array}
[c]{c}%
\xi\\
\eta
\end{array}
\right]  :\xi\in E,\ \eta\in L^{2}\left(  \left[  -h,0\right]  ,E\right)
,\ \frac{d}{d\theta}\eta\in L^{2}\left(  \left[  -h,0\right]  ,E\right)
,\ \eta\left(  0\right)  =\xi\right\}  .
\]%
\[
\widetilde{y}=\left[
\begin{array}
[c]{c}%
y\\
\overline{y}%
\end{array}
\right]  ,\ \widetilde{y}_{0}=\left[
\begin{array}
[c]{c}%
\xi\\
\eta
\end{array}
\right]  ,\ \ g\left(  \widetilde{y}\right)  =\left[
\begin{array}
[c]{c}%
\int_{\delta}^{b}h\left(  s,y\left(  s\right)  \right)  ds\\
0
\end{array}
\right]  ,\ \ \widetilde{B}=\left[
\begin{array}
[c]{c}%
B\\
0
\end{array}
\right]  ,\ \ \widetilde{f}=\left[
\begin{array}
[c]{c}%
f\\
0
\end{array}
\right]  .
\]
The system (\ref{ss1}) can be written in the standard form%
\begin{align}
\widetilde{y}^{\prime}\left(  t\right)   &  =\widetilde{A}\widetilde{y}\left(
t\right)  +\widetilde{B}u\left(  t\right)  +\widetilde{f}\left(
t,\widetilde{y}\left(  t\right)  \right)  ,\label{ss2}\\
\widetilde{y}\left(  0\right)   &  =\widetilde{y}_{0}+g\left(  \widetilde
{y}\right)  .\nonumber
\end{align}
Following \cite{delfour}, we say that the system (\ref{ss2}) is

\begin{itemize}
\item (partial-)approximately controllable on $\left[  0,T\right]  $ if
$\overline{\left\{  y\left(  T;u\right)  :u\in L^{2}\left(  \left[
0,T\right]  ,U\right)  \right\}  }=E,$

\item approximately $M^{2}$-controllable on $\left[  0,T\right]  $ if
$\overline{\left\{  \widetilde{y}\left(  T;u\right)  :u\in L^{2}\left(
\left[  0,T\right]  ,U\right)  \right\}  }=M^{2}\left(  \left[  -h,0\right]
,E\right)  .$
\end{itemize}

Assume that $\Phi\left(  t\right)  $ is a solution of the equation%
\begin{align*}
\Phi^{\prime}\left(  t\right)   &  =A\Phi\left(  t\right)  +N\Phi\left(
t-h\right)  +\int_{-h}^{0}M\left(  \theta\right)  \Phi\left(  t+\theta\right)
d\theta,\ t\geq h,\\
\Phi^{\prime}\left(  t\right)   &  =A\Phi\left(  t\right)  ,\ \ t<h,\ \ \Phi
\left(  0\right)  =I.\
\end{align*}
It is shown in \cite{delfour}, \cite{bm2} that the linear system corresponding
to (\ref{ss2}) is partial-approximately controllable if%
\[
B^{\ast}\Phi^{\ast}\left(  t-t\right)  y=0,\ 0<t<T\Longrightarrow y=0.
\]
Hence, the system (\ref{ss2}) is partial-approximately controllable on
$\left[  0,T\right]  $ provided that all conditions of Theorem \ref{thm:main}
are satisfied.

\section{Conclusion}

In this paper, partial-approximate controllability of semilinear evolution
systems in Hilbert spaces with nonlocal conditions have been investigated.
Sufficient condition for the partial-approximate controllability of such
systems have been established. Compared with some existing results, it can be
found that the variational approach, together with the approximating
technique, has been extended to consider the partial-approximate
controllability of more general systems.

The key hypothesis on the nonlocal function $g$ is the assumption (G), which
means that $g$ depends only on the value of $y$ in the interval $\left[
\delta,b\right]  $, $\delta>0$. This hypothesis covers the situation that the
nonlocal function $g$ is given by $g\left(  y\right)  =%
{\displaystyle\sum\limits_{k=1}^{p}}
c_{k}y\left(  t_{k}\right)  $, where $0<t_{1}<...<t_{p}<b,$ $c_{1},...c_{p}$
are given constants or by $g\left(  y\right)  =\int_{\delta}^{b}h\left(
s,y\left(  s\right)  \right)  ds$, $0<\delta<b$. However, if the nonlocal
function $g$ depends on the value of $y$ on the whole interval $\left[
\delta,b\right]  $, such as the nonlocal function $g$ is given by $g\left(
y\right)  =\int_{0}^{b}h\left(  s,y\left(  s\right)  \right)  ds,$ then even
the question of existence of a mild solution and corresponding controllability
issues are still open. For this reason, we are committed to studying such
problems in the future.

\textit{\bigskip}


\bigskip

\begin{thebibliography}{99}                                                                                               %


\bibitem {arti}G. Arthi, Ju H. Park , H.Y. Jung, Existence and controllability
results for second-order impulsive stochastic evolution systems with
state-dependent delay Applied Mathematics and Computation 248 (2014) 328--341.

\bibitem {arti2}G. Arthi. K. Balachandran, Controllability of impulsive
second-order nonlinear systems with nonlocal conditions in Banach spaces, J.
Control Decis. 2 (2015) 203--218.

\bibitem {nonloc6}P. Balasubramaniam, V. Vembarasan, T. Senthilkumar,
Approximate controllability of impulsive fractional integro-differential
systems with nonlocal conditions in Hilbert space. Numer. Funct. Anal. Optim.
35 (2014) 177--197.

\bibitem {bm}A. E. Bashirov, N. I. Mahmudov, On concepts of controllability
for deterministic and stochastic systems. SIAM J. Control Optim. 37 (1999) 1808--1821.

\bibitem {bm2}A. E. Bashirov, N. I. Mahmudov, N. \c{S}emi, H. Etikan, Partial
controllability concepts. Internat. J. Control 80 (2007) 1--7.

\bibitem {curt}R. F. Curtain, H. J. Zwart, An Introduction to Infinite
Dimensional Linear Systems Theory\textit{,}\ Springer-Verlag/New York, 1995.

\bibitem {dm1}J. P. Dauer, N. I. Mahmudov, Approximate controllability of
semilinear functional equations in Hilbert spaces. J. Math. Anal. Appl. 273
(2002) 310--327.

\bibitem {nonloc7}A. Debbouche, D.F.M. Torres, Approximate controllability of
fractional nonlocal delay semilinear systems in Hilbert spaces. Internat. J.
Control 86 (2013) 1577--1585.

\bibitem {nonloc4}A. Debbouche, D.F.M. Torres, Approximate controllability of
fractional delay dynamic inclusions with nonlocal control conditions. Appl.
Math. Comput. 243 (2014) 161--175.

\bibitem {delfour}M. C. Delfour, S. K. Mitter, Controllability, observability
and optimal feedback control of affine hereditary differential systems, SIAM
J. Control, 10, (1972) 298--328.

\bibitem {nonloc5}S. Ji, Approximate controllability of semilinear nonlocal
fractional differential systems via an approximating method. Appl. Math.
Comput. 236 (2014), 43--53.

\bibitem {nonloc2}F. D. Ge, H. C. Zhou, C. H. Kou, Approximate controllability
of semilinear evolution equations of fractional order with nonlocal and
impulsive conditions via an approximating technique. Appl. Math. Comput. 275
(2016) 107--120.

\bibitem {kilbas}A. A. Kilbas, H. M. Srivastava, J. J. Trujillo, Theory and
Applications of Fractional Differential Equations, Elsevier Science B.V.,
Amsterdam , 2006.

\bibitem {nonlocal8}J, Liang, J. H. Liu, T. J. Xiao, Nonlocal Cauchy problems
governed by compact operator families. Nonlinear Anal. 57 (2004) 183--189.

\bibitem {mah1}N. I. Mahmudov, Approximate controllability of semilinear
deterministic and stochastic evolution equations in abstract spaces. SIAM J.
Control Optim., 42 (2003) 1604--1622.

\bibitem {mahACM2}N. I. Mahmudov, Finite-approximate controllability of
evolution equations. Appl. Comput. Math. 16 (2017), 159--167.

\bibitem {mah2}N. I. Mahmudov, Approximate controllability of evolution
systems with nonlocal conditions, Nonlinear Anal., 68 (2008) 536-546.

\bibitem {mahmck}N. I. Mahmudov, M. A. McKibben, On approximately controllable
systems (survey). Appl. Comput. Math. 15 (2016) 247--264.

\bibitem {nai1}K. Naito, Approximate controllability for trajectories of
semilinear control systems, J. Optim. Theory Appl., 60 (1989) 57--65.

\bibitem {sukavanam1}N. Sukavanam, S. Tafesse, Approximate controllability of
a delayed semilinear control system with growing nonlinear term, Nonlinear
Anal. TMA .74 (2011) 6868- 6875.

\bibitem {pazy}A. Pazy, Semigroups of Linear Operators and Applications to
Partial Differential Equations, Springer Verlag, New York, 1983.

\bibitem {pod}I. Podlubny, Fractional Differential Equations, Academic Press,
San Diego, 1999.

\bibitem {mah6}R. Sakthivel, Y. Ren, N. I. Mahmudov, On the approximate
controllability of semilinear fractional differential systems, Comput. Math.
Appl. .62 (2011) 451-1459.

\bibitem {sakthivel}R. Sakthivel, Y. Ren, Approximate controllability of
fractional di erential equations with state-dependent delay, Results Math., 63
(2013), 949\{963.

\bibitem {seid}T. I. Seidmann, Invariance of the reachable set under nonlinear
pertubations, SIAM J. Control and Optim., 25 (1985) 1173-1191.

\bibitem {nonloc3}X. Zhang, C. Zhu, C. Yuan, Approximate controllability of
fractional impulsive evolution systems involving nonlocal initial conditions.
Adv. Difference Equ. 2015, 2015:244, 14 pp.

\bibitem {zhou}H. X. Zhou, Approximate controllability for a class of
semilinear abstract equations, SIAM Journal of Control and Optim., 21 (1983) 551-565.

\bibitem {zhou1}Y. Zhou, F. Jiao, Nonlocal Cauchy problem for fractional
evolution equations, Nonlinear Anal. RWA 11 (2010) 4465--4475.

\bibitem {zhoubook}Y. Zhou, Basic Theory of Fractional Di erential Equations,
World Scientifc, Singapore, 2014.

\bibitem {yamato}M. Yamamoto, J. Y. Park, Controllability for parabolic
equations with uniformly bounded nonlinear terms,\textit{\ }J. Optim. Theory
Appl., 66 (1990) 515-532.

\bibitem {zuazua4}E. Zuazua, Finite dimensional null controllability for the
semilinear heat equation. J. Math. pures et appl. , 76 (1997) 570-594.
\end{thebibliography}
\end{document}